\documentclass[12pt]{amsart}
\usepackage{amsfonts, amssymb, amsmath, amscd, amsthm, txfonts, graphicx, color, enumerate}

\newtheorem{teo}{Theorem}[section]
\newtheorem{propo}[teo]{Proposition}
\newtheorem{cor}[teo]{Corollary}
\newtheorem{lem}[teo]{Lemma}

\newtheorem{defin}[teo]{Definition}
 
\newtheorem{obs}[teo]{Remark}
\numberwithin{equation}{section}
\newtheorem{ej}[teo]{Example}

\begin{document}

\title[Continuity    Topological Entropy of Non-stationary Dynamical Systems]{On the Continuity of the  Topological Entropy of Non-autonomous Dynamical Systems}

\author[J. Muentes]{Jeovanny de Jesus Muentes Acevedo}
\address{Instituto de Matem\'atica e Estat\'istica\\ Universidade de S\~ao Paulo\\
05508-090, Sao Paulo, Brazil}
\email{jeovanny@ime.usp.br}

\date{2017}

\begin{abstract} Let $M$ be a compact Riemannian manifold.  The set $\text{F}^{r}(M)$ consisting of sequences $(f_{i})_{i\in\mathbb{Z}}$ of $C^{r}$-diffeomorphisms on $M$ can be endowed with the compact topology or with the strong topology. A notion of topological entropy is given for these   sequences.   I will prove this entropy is discontinuous at each sequence  if we consider the compact topology on $\text{F}^{r}(M)$. On the other hand, if $ r\geq 1$ and we consider the strong topology on $\text{F}^{r}(M)$, this entropy is a continuous map.       
\end{abstract}



\maketitle

\section{Introduction}
In 1965, R. L. Adler, A. G. Konheim and M. H. McAndrew introduced the topological entropy
of a continuous map $\phi : X\rightarrow X$  on a compact topological space X via open covers of $X$.  
Roughly, topological entropy   is the exponential growth rate of the number of essentially different orbit segments of length $n$.  
In 1971,   R. Bowen defined the topological entropy   of a uniformly continuous map  	 on an arbitrary metric space   via  spanning
and separated sets, which, when the space is compact,   it coincides with the topological entropy as
defined by Adler, Konheim and McAndrew.
Both definitions can be found in  \cite{Walters}.

\medskip

Let $M$ be a compact metric space. Let     $\textbf{\textit{f}}=(f_{i})_{i\in\mathbb{Z}}$ be a sequence of homeomorphisms defined on $M.$   The $n$-th composition is defined, for each   $i\geq 1$,    as  \[\textbf{\textit{f}}_{ i} ^{n}=  f_{i+n-1}\circ \cdots\circ f_{i} \quad\text{and}\quad \textbf{\textit{f}}_{ i} ^{-n}=f_{i-n}^{-1}\circ \cdots\circ f_{i-1}^{-1}:M_{i}\rightarrow M_{i-n} , \quad n\geq 0.\] 
This notion is known as    \textit{non-stationary dynamical systems} or \textit{non-autonomous dynamical systems} (see \cite{alb}, \cite{Kloeden}, \cite{Kolyada}).  
S. Kolyada and L. Snoha, in \cite{Kolyada}, introduced a notion of topological entropy for this type of dynamical systems, which generalizes the notion of entropy for single dynamical systems.  They only considered    sequences of type $(f_{i})_{i\geq0}$ and the entropy for this sequence is a single number (possibly $+\infty$).   Naturally,   this idea can be extended to two-sided sequences   $(f_{i})_{i\in\mathbb{Z}}$.   We will consider   sequences of type $(f_{i})_{i\in\mathbb{Z}}$ because, since each $ f_ {i}$ is a homeomorphism, we can compute another entropy for the same sequence by considering the composition of the inverse of each $f_{i}$ for $i\rightarrow -\infty$ (see Remark \ref{obsentropiainversa}).

 \medskip
 
 Firstly,  the entropy of a non-autonomous dynamical system $(f_{i})_{i\in\mathbb{Z}}$ will be defined as a sequence of non-negative numbers $ (a_{i})_{i \in \mathbb{Z}} ,$ where each $a_{i}$ depends only on $f_{j}$ for $j\geq i.$   Then we will see that    $ (a_{i})_{i \in \mathbb{Z}}$ is a constant sequence (see Corollary \ref{corolarioigualdad}).   Consequently, this common number will be considered as the entropy of $(f_{i})_{i\in\mathbb{Z}}$.  As a consequence, we  will also see  the entropy of a non-autonomous dynamical system can be considered as the topological entropy of a single homeomorphism defined on the union disjoint of   infinitely many copies of $ M $ (see Remark \ref{remarkunmap}).

\medskip

     Let $\text{F}^{r}(M)$ be the set consisting of families $(f_{i})_{i\in\mathbb{Z}}$ of $C^{r}$-diffeomorphisms on $M$, where $M$ is a compact Riemannian manifold. $\text{F}^{r}(M)$ can be endowed with the compact topology and the strong topology (see Definitions \ref{producttopo} and \ref{strongtopo}). In this paper I will show that, if $r\geq1$ and if we consider the strong topology on $\text{F}^{r}(M)$, the entropy depends   continuously  on each sequence in $\text{F}^{r}(M)$.    In contrast, with the product topology on $\text{F}^{r}(M)$, the entropy is discontinuous at any sequence.

\medskip

 Many results are well-known about the continuity of the entropy  of single maps.  In
\cite{Newhouse},   Newhouse proved that the topological entropy of $C^{\infty}$-diffeomor\-phisms  on a compact Riemannian manifold    is an
 upper semicontinuous map. Furthermore, if $M$ is a surface, this map is continuous. The entropy for any homeomorphism on the circle $\mathbb{S}^{1}$ is zero. Therefore, it depends continuously on homeomorphisms on $\mathbb{S}^{1}$. In contrast, the entropy does not depend on continuous maps that are not   homeomorphisms on the circle (see \cite{Block}).  About the continuity of the entropy for flows, readers could see \cite{Jiagang}.

\medskip

Next, I will   talk about the structure of this work. In the next section I will give a motivation for this work as well as some further generalizations.  Section 3 will be devoted to remembering the notions of non-autono\-mous dynamical systems. Furthermore we will also   see the type of conjugacies that work for these   systems and   the strong   and product topologies on $\text{F}^{r}(M)$.  In Section 4 will be introduced the entropy for non-autonomous dynamical systems. This will be given using both open partitions of $M$   and also   separated and spanning  sets. These definitions coincide, as in the case of single maps.  Some properties of the entropy will be given in Section 5. These properties generalize to the ones of the entropy of single maps.   Finally, in Section 6, we will see that the entropy is continuous  on $\text{F}^{r}(M)$ with the strong topology if $r\geq 1$. More specifically, it  is locally constant.    In contrast, it is discontinuous at any sequence if we consider the product topology on $\text{F}^{r}(M)$. 

\medskip

The present work was carried out with the support of   the Conselho Nacional de Desenvolvimento Cient\'ifico  e Tecnol\'ogico - Brasil (CNPq) and the Coordena\c c\~ao de Aperfei\c coamento de Pessoal de N\'ivel Superior  (CAPES).

\section{Motivation and Further Generalizations}
 Dynamical systems are classified via topological  conjugacies. 
Uniform conjugacies (see Definition \ref{definconjugacy}) are very  suitable for classifying non-autotomous dynamical systems,  time-one maps of flows, discrete time process generated by non-autonomous differential equations, among others 
 systems (see \cite{alb}, \cite{Kloeden}, \cite{Jeo1}, \cite{Jiagang}). In that case,  the entropy  plays a fundamental role, since it is invariant by uniform conjugacies (see  Theorem \ref{invarianteentropia}). In \cite{alb}, \cite{Jeo1} and \cite{Jeo3} can be found several properties that are invariant by uniform conjugacies. 
 
 \medskip

 Next, considering the product topology on $\text{F}^{r}(M)$, the entropy for non-autono\-mous dynamical systems could be a new tool to study the continuity of the topological entropy for some single maps (see Proposition \ref{propodeee}).

\medskip

 The entropy to be constructed here, it will be fixed a metric space $ M $ and each map $ f_{i}:M\rightarrow M$ will be a homeomorphisms. This notion can be extended considering, for each $i\in\mathbb{Z}$, a more general metric space  $M_{i}$   (that is, $M_{i}$ must not necessarily be of the form $M\times \{i\}$, as will be considered in this work)  with a fixed metric $d_i$ and each  $f_i $ being a continuous map on $M_{i}$ to $ M_{i+1}$, not necessarily a homeomorphism.   An interesting work would be to study the properties of this entropy.

\medskip

In this work will be proved  the continuity of the entropy of non-autonomous dynamical systems as long as each diffeomorphism $f_{i}$ is of class $ C^{r}$ with $ r \geq1 $.   Another very interesting work would be to study the continuity of this entropy for  sequences of H\"older continuous homeomorphisms. In this case, $M$ could be a general metric space, that is, not necessarily a differentiable manifold, and the continuity could depend on the Hausdorff dimension of $M$.    A series of results that could be very useful to work on this problem can be found in  
\cite{Block}, \cite{Kolyada},    \cite{Xionping}, among others  papers.

\section{Non-autonomous Dynamical Systems, Uniform Conjugacy and Strong Topology}
Given a   metric space    $M$   with metric $d$, consider the  \textit{disjoint union}  $$\textbf{M}=\coprod_{i\in \mathbb{Z}}{M_{i}}=\bigcup_{i\in \mathbb{Z}}{M\times{i}}.$$ The set $\textbf{M}$ will be called   \textit{total space} and the $M_{i}$ will be called  \textit{components}.    
Remember that a subset   $A\subseteq \textbf{M}$ is open in  $\textbf{M}$ if only if  $A\cap M_{i}$ is open in  $M_{i}$, for all $i\in \mathbb{Z}$. 
The total space will be equipped with the metric
\begin{equation}\label{metricatotal1}
\textbf{d}(x,y)=    \begin{cases}
        \min\{1,d(x,y)\} & \mbox{if }x,y\in M_{i} \\
        	1 & \mbox{if }x\in M_{i}, y\in M_{j} \mbox{ and }i\neq j. \\
        \end{cases} 
\end{equation}

 Two metrics   $\rho_{1}$ and  $\rho_{2}$ on a   topological space $X$  are   \textit{uniformly equivalent}   if there exist   positive numbers  $k$ and $K$ such that $k\rho_{1} (x,y)\leq \rho_{2}(x,y)\leq K\rho_{1} (x,y)$ for all $x,y\in X$. It is clear that if  $\hat{d}$ and  $\tilde{d}$ are   uniformly equivalent metrics     on $ M$, then,       $\hat{\textbf{d}} $ and $\tilde{\textbf{d}}$, obtained as in      \eqref{metricatotal1}, generate the same topology  on \textbf{M} and, in that case, they  are uniformly equivalent on \textbf{M}.  On the other hand, if $\hat{d}_{i}$ and $\tilde{d}_{i}$ are uniformly   equivalent metrics on $ M_{i}$ for each $i\in \mathbb{Z}$,   then the metrics  $\hat{\textbf{d}} $ and  $\tilde{\textbf{d}} $, defined similarly   as in  \eqref{metricatotal1}, generate the same topology on the total space, but they are not necessarily  uniformly  equivalent  on \textbf{M} (notice that \textbf{M} is not compact). Throughout
 this work, we fix a metric $d$ on $M$ and we consider the metric on the total space as it was defined in \eqref{metricatotal1}.  Without losing generality, we can suppose that the diameter of $M$ is less than or equal to 1 (in that case, if $x,y\in M_{i}$, then $\textbf{d}(x,y)=d(x,y)$).

\begin{defin}\label{leidecomposicao} A  \textit{non-autonomous dynamical system} \textbf{\textit{f}} on \textbf{M}, which will de denoted by $(\textbf{M}, \textbf{\textit{f}})$,  is an application $\textbf{\textit{f}}:\textbf{M}\rightarrow \textbf{M}$, such that, for each $i\in\mathbb{Z}$, $\textbf{\textit{f}}|_{M_{i}}=f_{i}:M_{i}\rightarrow M_{i+1}$ is a  homeomorphism. Sometimes we use the notation   $\textbf{\textit{f}}=(f_{i})_{i\in\mathbb{Z}}$. A $n$-th composition is defined, for each   $i\in \mathbb{Z}$,    as   
\begin{equation*}\textbf{\textit{f}}_{ i} ^{n}:= \begin{cases}
   f_{i+n-1}\circ \cdots\circ f_{i}:M_{i}\rightarrow M_{i+n}  & \mbox{if }n>0 \\
  f_{i-n}^{-1}\circ \cdots\circ f_{i-1}^{-1}:M_{i}\rightarrow M_{i-n}  & \mbox{if }n<0 \\
	I_{i}:M_{i}\rightarrow M_{i} & \mbox{if }n=0,\\
        \end{cases}
\end{equation*}
where $I_{i}$ is the identity on $M_{i}$.
\end{defin}

A simple example of a  non-autonomous dynamical systems  is the \textit{constant family} associated to a homeomorphism:
 
\begin{ej} \label{levantamento} Let  $\phi:M\rightarrow M$ be a   homeomorphism.   The \textit{constant family} $(\textbf{M}, \textbf{\textit{f}})$ associated to $\phi$ is the sequence $(f_{i}:M_{i}\rightarrow M_{i+1})_{i\in\mathbb{Z}}$   defined as $f_{i}(x,i)=(g(x),i+1)$ for each $x\in M$ and $i\in\mathbb{Z}$. 
\end{ej}

Next we talk about the morphisms between non-autonomous dynamical systems. Take $$\textbf{N}=\coprod_{i\in \mathbb{Z}}{N\times\{i\}},$$ where  $N $ is a   metric space, and consider  a non-autonomous dynamical system  \textbf{\textit{g}} defined on \textbf{N}. 
A \textit{topological conjugacy} between    $(\textbf{M},\textbf{\textit{f}})$ and $(\textbf{N},\textbf{\textit{g}})$  
is a map $\textbf{\textit{h}}:\textbf{M}\rightarrow \textbf{N}$, such that, for  each $i\in \mathbb{Z} ,$ $\textbf{\textit{h}}|_{M_{i}}=h_{i}:M_{i}\rightarrow N_{i}$ is a homeomorphism     and   \(h_{i+1}\circ f_{i}=g_{i}\circ h_{i}:M_{i}\rightarrow N_{i+1},\)  
that is, the following diagram commutes: 
\[\begin{CD}
M_{-1}@>{f_{-1}}>> M_{0}@>{f_{0}}>>M_{1}@>{f_{1}}>>M_{2} \\
@V{\cdots}V{h_{-1}}V @VV{h_{0}}V @VV{h_{1}}V @VV{h_{2}\cdots}V\\
N_{-1}@>{g_{-1}}>> N_{0}@>{g_{0}}>>N_{1} @>{g_{1}}>>N_{2}
\end{CD} 
\]
It is clear that the  topological conjugacies define a  equivalence relation on the set consisting of the non-autonomous dynamical systems on  $\textbf{M}$. However, if  $M_{0}$ and $N_{0}$ are homeomorphic,   the partition obtained by this relation is trivial: Indeed,  if $h_{0}$ is a homeomophism between $M_{0}$ and $N_{0}$,  the systems  $(\textbf{M} , \textbf{\textit{f}})$ and $(\textbf{N}, \textbf{\textit{g}})$ are conjugate  by $\textbf{\textit{h}}:\textbf{M}\rightarrow \textbf{N}$ defined as
\begin{equation*}\label{conjugacaotopologica} h_{i}=
\begin{cases}
  h_{0}  & \mbox{if }i=0 \\
  g_{i-1}\circ\cdots\circ g_{0}\circ h_{0}\circ f_{0} ^{-1}\circ\cdots\circ f_{i-1} ^{-1} & \mbox{if }i>0 \\
	g_{i}^{-1}\circ\cdots\circ g_{-1}^{-1}\circ h_{0}\circ f_{-1}\circ\cdots\circ f_{i} & \mbox{if }i<0.
        \end{cases} 
\end{equation*}
 
One type of conjugacy that works  for the class of non-autonomous dynamical systems are the \textit{uniform  conjugacy}:   
\begin{defin}\label{definconjugacy} We say that a topological conjugacy  $\textbf{\textit{h}}:\textbf{M}\rightarrow \textbf{N}$ between    $(\textbf{M},\textbf{\textit{f}})$ and $(\textbf{N},\textbf{\textit{g}})$   is  \textit{uniform}     if $(h_{i}:M_{i}\rightarrow N_{i})_{i\in\mathbb{Z}}$ and $(h_{i}^{-1}:N_{i}\rightarrow M_{i})_{i\in\mathbb{Z}}$ are   equicontinuous  sequences (that is, $\textbf{\textit{h}}$ and $\textbf{\textit{h}}^{-1}$ are uniformly continuous). In that case we  will say that the systems are  \textit{uniformly conjugate}. 
\end{defin}
 
Since the composition of   uniformly  continuous applications     is  uniformly continuous,
the  class consisting of non-autonomous dynamical systems becomes  a category, where the objects are the non-stationary dynamical systems and the morphisms are the   uniform conjugacies.

\medskip

Another notion of conjugacy, which is weaker than the conjugacy given in   Definition \ref{definconjugacy},  that  also  works for  non-autonomous dynamical systems   is the next one: 
\begin{defin} A \textit{positive (negative) uniform  conjugacy} between two systems  $(\textbf{M},  \textbf{\textit{f}})$ and $(\textbf{N},  \textbf{\textit{g}})$ is a sequence of homeomorphisms  $h_{i}:M_{i}\rightarrow N_{i}$ for $i\geq 0$ (for  $i\leq 0$) such that  $(h_{i})_{i\geq 0}$ and $(h_{i}^{-1})_{i\geq 0}$ ($(h_{i})_{i\leq 0}$ and $(h_{i}^{-1})_{i\leq 0}$) are equicontinuous  and   $h_{i+1}\circ f_{i}=g_{i}\circ h_{i}:M_{i}\rightarrow N_{i+1},$ for every $i\geq 0$ (for every $i\leq -1$).  That is, $(f_{i})_{i\geq 0}$ and $(g_{i})_{i\geq 0}$  ($(f_{i})_{i\leq 0}$ and $(g_{i})_{i\leq 0}$) are uniformly conjugate.
\end{defin} 

The following lemma it is clear and therefore we will omit the proof.
 
\begin{lem}\label{lemapositive}   $(\textbf{M},  \textbf{\textit{f}})$ and $(\textbf{N},  \textbf{\textit{g}})$ are positive (negative) uniformly conjugate if and only if, for any $i_{0}\in\mathbb{Z}$  there exists a  sequence of homeomorphisms    $(h_{i})_{i\geq i_0}$ ($(h_{i})_{i\leq i_0}$)  such that $(h_{i})_{i\geq i_0}$ and $(h_{i}^{-1})_{i\geq i_0}$ ($(h_{i})_{i\leq i_0}$ and $(h_{i}^{-1})_{i\leq i_0}$) are equicontinuous and $h_{i+1}\circ f_{i}=g_{i}\circ h_{i}:M_{i}\rightarrow N_{i+1},$ for every $i\geq i_0$ (for every $i\leq i_0$).
\end{lem}

Take two homeomorphisms    $g_{1}:X_{1}\rightarrow X_1$ and $g_{2}:X_{2}\rightarrow X_{2}$ defined on two  compact metric spaces  $X_{1}$ and $X_{2}$.  Let $\textbf{\textit{f}}_{1}$ and $\textbf{\textit{f}}_{2}$ be the    constant families associated, respectively,   to $g_{1}$ and to $g_{2}.$ It is clear that if $g_{1}$ and  $g_{2}$ are \textit{topologically conjugate} (i. e., there exists a homeomorphism $h:X_{1}\rightarrow X_{2}$ such that $h\circ g_{1}=g_{2}\circ h$) then $\textbf{\textit{f}}_{1}$ and $\textbf{\textit{f}}_{2}$ are  uniformly conjugate. In    \cite{Jeo3}  is   proved  the reciprocal is not always   true, that is, there exist uniformly conjugate  constant families $\textbf{\textit{f}}_{1}$ and $\textbf{\textit{f}}_{2}$,  associated, respectively,   to two homeomorphisms $g_{1}$ and   $g_{2}$ that are not topologically conjugate. 
 
\begin{defin}Let     $(\textbf{M},\textbf{\textit{f}})$ and  $(\widetilde{\textbf{M}},\tilde{\textbf{\textit{f}}})$  be  non-autonomous dynamical systems. We say that  $(\widetilde{\textbf{M}},\tilde{\textbf{\textit{f}}})$ is a \textit{gathering} of  $(\textbf{M},\textbf{\textit{f}})$ if there exists a strictly increasing  sequence   of integers $(n_{i})_{i\in\mathbb{Z}}$ such that $\widetilde{M}_{i}=M_{n_{i}}$ and $\tilde{\textbf{\textit{f}}}_{i}=f_{n_{i+1}-1}\circ \cdots \circ f_{n_{i}+1}\circ f_{n_{i}}$:  
\begin{equation*}\label{gathering}\begin{CD}
\cdots M_{n_{i-1}}@>{\tilde{f}_{i-1}=f_{n_{i}-1}\circ \cdots  \circ f_{n_{i-1}}}>> M_{n_{i}}@>{\tilde{f}_{i}=f_{n_{i+1}-1}\circ \cdots  \circ f_{n_{i}}}>>M_{n_{i+1}}   \cdots
\end{CD} 
\end{equation*}
If   $(\widetilde{\textbf{M}},\tilde{\textbf{\textit{f}}})$ is a    gathering of $(\textbf{M},\textbf{\textit{f}})$,  we say that $(\textbf{M},\textbf{\textit{f}})$ is a  \textit{dispersal} of $(\widetilde{\textbf{M}},\tilde{\textbf{\textit{f}}})$. 
\end{defin}
In \cite{alb}, Proposition 2.5,  is proved that any non-autonomous dynamical system  has a dispersal, which has a gathering,
which is equal to the constant family associated to the identity  on $M$.  Notice that, if $(\textbf{M},\textbf{\textit{f}})$ and $(\textbf{N},\textbf{\textit{g}})$ are uniformly   conjugate by   $\textbf{\textit{h}}=(h_{i})_{i\in\mathbb{Z}}$, then the gatherings 
 $(\widetilde{\textbf{M}},\tilde{\textbf{\textit{f}}})$ and  $(\widetilde{\textbf{N}},\tilde{\textbf{\textit{g}}})$   obtained,   respectively, of $(\textbf{M},\textbf{\textit{f}})$ and $(\textbf{N},\textbf{\textit{g}})$ by a sequence of integers $(n_i)_{i\in\mathbb{Z}}$, are   uniformly   conjugate by the    family $\tilde{\textbf{\textit{h}}}=(\tilde{h}_{n_{i}})_{i\in\mathbb{Z}}:$ 
 \[\begin{CD}
M_{n_{i-1}}@>{f_{n_{i-1}}}>>\cdots  @>{f_{n_{i}-1}}  >> M_{n_{i}}@>{f_{n_{i}}} >>\cdots @>{f_{n_{i+1}-1}} >> M_{n_{i+1}} \\
@V{\cdots}V{h_{n_{i-1}}}V @.   @VV{h_{n_{i}}}V @.  @VV{h_{n_{i+1}}\cdots}V \\
M_{n_{i-1}}@>{g_{n_{i-1}}}>>\cdots  @>{g_{n_{i}-1}}  >> M_{n_{i}}@>{g_{n_{i}}}>>\cdots @>{g_{n_{i+1}-1}} >> M_{n_{i+1}}
\end{CD} 
\]

\medskip

 I will finish this section giving two different topologies   to the space consisting of non-autonomous dynamical systems: the \textit{product topology}   and the   \textit{strong topology}. All the results and notions that will be presented in this part of the work can be found in  \cite{Hirsch}.  Let $r\geq0$. I will suppose that $M$ is a compact $C^{r}$-Riemannian manifold with Riemannian norm $\Vert\cdot\Vert$. This norm induces a metric $d$ on $M$. Set  \[\text{F}^{r}(\textbf{M})=\{\textbf{\textit{f}}=(f_{i})_{i\in\mathbb{Z}}: f_{i}:M_{i}\rightarrow M_{i+1} \text{ is a  } C^{r}\text{-diffeomorphism}\}.\] 
If $r=0$, $\text{F}^{r}(\textbf{M})$ consists of the  sequences of homeomorphisms.  
Let \[\text{Diff}^{r}(M_{i},M_{i+1})= \{ g:M_{i}\rightarrow M_{i+1}: g \text{ is a } C^{r}\text{-diffeomorphism}\}.\] 
 The Riemannian metric $\Vert\cdot\Vert$ induces a $C^{r}$-metric on   $\text{Diff}^{r}(M_{i},M_{i+1})$, which will be denoted by $d^{r}$.  
 Notice that \[ \text{F}^{r}(\textbf{M})=\prod_{i=-\infty} ^{+\infty}  \text{Diff}^{r}(M_{i},M_{i+1}) .\]
 
\begin{defin}\label{producttopo}   The \textit{product topology} on $\text{F}^{r}(\textbf{M})$ is generated by the sets    
\[ \mathcal{U}=\prod_{i<-j}   \text{Diff}^{r}(M_{i},M_{i+1})\times \prod_{i=-j}^{j}[U_{i} ]\times \prod_{i>j}   \text{Diff}^{r}(M_{i},M_{i+1}), \]
where $U_{i}$ is an open subset of $\text{Diff}^{r}(M_{i},M_{i+1})$, for $-j\leq i\leq j,$ for some  $j\in \mathbb{N}$.  
The space $\text{F}^{r}(\textbf{M})$ with the product topology will be denoted by $(\text{F}^{r}(\textbf{M}),\tau_{prod}).$ 
\end{defin}

\begin{defin}\label{strongtopo}     For each  $\textbf{\textit{f}}\in \text{F}^{r}(\textbf{M})$ and a sequence of positive numbers   $\varepsilon=(\varepsilon_{i})_{i\in \mathbb{Z}}$, a  \textit{strong basic neighborhood} of $\textbf{\textit{f}}$ is the set \[ B^{r}(\textbf{\textit{f}},\varepsilon)=  \left\{\textbf{\textit{g}}=(g_{i})_{i\in \mathbb{Z}}\in \text{F}^{r}(\textbf{M}):    d^{r} (f_{i},g_{i}) <\varepsilon_{i},\text{ for all  }  i\in\mathbb{Z}\right\}. \]  
 The $C^{r}$-\textit{strong topology} (or $C^{r}$-\textit{Whitney topology}) on $\text{F}^{r}(\textbf{M})$ is generated by the strong basic neighborhoods  of each $\textbf{\textit{f}}\in \text{F}^{r}(\textbf{M})$. The space $\text{F}^{r}(\textbf{M})$ with the strong topology will be denoted by $(\text{F}^{r}(\textbf{M}),\tau_{str}).$ 
 \end{defin}

Notice that $\tau_{str}$ is finer than   $\tau_{prod}$ on $\text{F}^{r}(\textbf{M})$, that is, 
\begin{align*} I : (\text{F}^{r}(\textbf{M}),\tau_{str})&\rightarrow (\text{F}^{r}(\textbf{M}),\tau_{prod})\\
(f_{i})_{i\in\mathbb{Z}}& \mapsto (f_{i})_{i\in\mathbb{Z}}  
\end{align*}
is   continuous.

\section{Entropy for Non-stationary Dynamical Systems}

  In this section we will see how the topological entropy for a non-autonomous dynamical systems $(f_{i})_{i\in\mathbb{Z}}$ is constructed.       Firstly, this  entropy  will not be    a real number, but a   sequence of non-negative numbers (possibly $+\infty$) $(a_{i})_{i\in\mathbb{Z}}$ where each $a_{i}$ depends only on $f_{j}$ for $j\geq i.$    In the next section will be proved that this sequence is constant. Consequently, this common value will be considered as the topological  entropy 
of   $(f_{i})_{i\in\mathbb{Z}}$. The proof  of the   statements  in this section can be found in \cite{Walters} for the case of a  single map. Such proofs  can be adapted for non-stationary dynamical systems and, therefore, will be omitted.  

\medskip

In order to define the entropy,    consider the following notions:  an \textit{open cover} of  $M$ is a collection of open subsets of $M$, $\mathcal{A}=\{A_{\lambda}\}_{\lambda\in\Lambda},$ such that $M=\bigcup _{\lambda} A_{\lambda}.$ In this section,  $\mathcal{A}$ and $\mathcal{B}$ will denote   open covers of $M$. Since $M_{i}=M\times \{i\},$ if  $\mathcal{A}$ is an open  cover of $M$, then $\mathcal{A}_{i}=\mathcal{A}\times \{i\} $ is an open cover of $M_{i}$. By abuse of notation,  I will omit the sub index $i$ of  $\mathcal{A}_{i}$  for covers of  $M_{i}$. 
 
\begin{defin}\label{cee} Let $N(\mathcal{A})$ be the number of sets in a finite subcover of  $\mathcal{A}$ with smallest cardinality.  The \textit{entropy} of $\mathcal{A}$ is the number   $H(\mathcal{A}):=\text{log} N(\mathcal{A})$. 
\end{defin}

For each $i\in\mathbb{Z}$   and $n\geq 0,$   set
   $(\textbf{\textit{f}}_{ i} ^{n})^{-1}(\mathcal{A})=\{(f_{i+n-1}\circ\cdots\circ f_{i})^{-1}(A):A\in \mathcal{A}\}$. Set  $\mathcal{A}  \vee \mathcal{B}=\{A\cap B:A\in \mathcal{A}, B\in \mathcal{B}\}.$  Inductively we can define $\bigvee_{m=1}^{k}\mathcal{A}^{m}$ for a collection of open covers $\mathcal{A}^{1},..., \mathcal{A}^{k}$ of $M.$     $\mathcal{B}$ is a \textit{refinement} of $ \mathcal{A}$ if each element of $ \mathcal{B} $ is contained in some element of $ \mathcal{A}$.

\begin{propo}\label{propoentropia} The entropy satisfies the following properties:
\begin{enumerate}
\item $H (\mathcal{A}\vee \mathcal{B})\leq H(\mathcal{A})+ H(\mathcal{B}).$ 
\item If $\mathcal{B}$ is a refinement    of $ \mathcal{A}$ then $H(\mathcal{A})\leq H (\mathcal{B})$.
\item $H(\mathcal{A})=H((\textbf{f}_{ i} ^{\, k})^{-1}(\mathcal{A}))$  for each $i\in\mathbb{Z}$ and $k\geq0.$ 
\item $H(\bigvee_{k=0}^{n-1}(\textbf{f}_{ i} ^{\, k})^{-1}(\mathcal{A})) \leq n H(\mathcal{A})$, for each $i\in\mathbb{Z}$ and $n\geq1.$ 
\item The limit  
\begin{equation} \label{limiteentropia} 
H_{i}(\textbf{f},\mathcal{A})=\lim_{n\rightarrow +\infty} \frac{1}{n}H\left(\bigvee_{k=0}^{n-1}
(\textbf{f}_{ i} ^{ \, k})^{-1}(\mathcal{A})\right)
\end{equation}
exists and is  finite, for each $i\in\mathbb{Z}$. 
\end{enumerate}
\end{propo} 

\begin{defin}\label{entropiacover} We define the  \textit{entropy} of  $\textbf{\textit{f}}$ \textit{relative to} $\mathcal{A}$ as the sequence  $\mathcal{H}(\textbf{\textit{f}},\mathcal{A})=(H_{i}(\textbf{\textit{f}},\mathcal{A}))_{i\in\mathbb{Z}}$.
The \textit{topological entropy of} $\textbf{\textit{f}}$ is the sequence  $\mathcal{H}(\textbf{\textit{f}})=(\mathcal{H}_{i}(\textbf{\textit{f}}))_{i\in\mathbb{Z}}$, where \[\mathcal{H}_{i}(\textbf{\textit{f}})=\sup \{H_{i}(\textbf{\textit{f}},\mathcal{A}):\mathcal{A} \text{ is an open cover of }M\}.\]
\end{defin} 

From now on, $X$ will represent a compact metric space. We recall   the \textit{topological entropy} of a homeomorphism $ g: X \rightarrow X $, which we denote by $ h (g) $, is defined considering open covers of   $ X $.
Definition \ref{entropiacover} only makes sense when   $\mathcal{A}$ is an  open cover of $M$ instead of a general  open cover of   \textbf{M}. If we consider arbitrary collections of open covers of each $M_{i}$,   the limit  \eqref{limiteentropia} could be infinite (we   can take open covers  $\mathcal{A}_{i}$ of each $M_{i}$ with   $N(\mathcal{A}_{i})$ arbitrarily large, for each $i$).

\medskip

 Now we introduce the definition of   topological entropy using spanning  and separated subsets. That entropy will be called   $\star$-\textit{topological entropy}  for  differentiate it from the topological entropy. As in the case of a single homeomorphism,   the topological entropy coincides with $\star$-topological entropy for sequences (see   Theorem \ref{igualdadeentropias}).
 
\begin{defin}Let $n\in\mathbb{N}$, $\varepsilon>0$ and $i\in\mathbb{Z}$ be given. We say that a compact subset $K\subseteq M_{i}$ is a  $(n,\varepsilon)$-\textit{span} of  $M_{i}$  \textit{with respect} \textbf{\textit{f}} if for each $x\in M_{i}$ there exists $y\in K$ such that $\max_{0\leq j< n}\textbf{d}( \textbf{\textit{f}}_{i} ^{j}(x), \textbf{\textit{f}}_{i} ^{j}(y))<\varepsilon$, i. e., \(M_{i}\subseteq \bigcup_{y\in K} \bigcap _{k=0}^{n-1}(\textbf{\textit{f}}_{i}^{\,k})^{-1}(\overline{B(\textbf{\textit{f}}_{i}^{\,k}(y),\varepsilon)}),\)
where $ B(\textbf{\textit{f}}_{i}^{\,k}(y),\varepsilon)$ is the open ball with center   $\textbf{\textit{f}}_{i}^{\,k}(y)\in M_{i+k}$ and radius $\varepsilon$.\end{defin} 
Denote by $r[n,i](\varepsilon,\textbf{\textit{f}})$ the smallest cardinality of any $(n,\varepsilon)$-span of $M_{i}$ with respect \textbf{\textit{f}}. Since $M_{i}$ is compact, we have $r[n,i](\varepsilon,\textbf{\textit{f}})<\infty$ for each $i\in\mathbb{Z}$ and $n\geq1.$  
Set \( r[i](\varepsilon,\textbf{\textit{f}})=\underset{n\rightarrow +\infty}\limsup  \frac{1}{n}\log r[n,i](\varepsilon,\textbf{\textit{f}}).\)

\begin{defin}\label{entropiametrica} The $\star$-\textit{topological entropy of} \textbf{\textit{f}} is the sequence $\textbf{H}(\textbf{\textit{f}})=(\textbf{H}_{i}(\textbf{\textit{f}}))_{i\in\mathbb{Z}}$ given by
$$\textbf{H}_{i}(\textbf{\textit{f}})= \lim_{\varepsilon\rightarrow 0}r[i](\varepsilon,\textbf{\textit{f}}).$$
\end{defin}
Now we define the entropy for families using separated subsets and we will prove that the   entropy  considering span subsets   coincide with the entropy considering separated subsets.  
\begin{defin} Let $n\in\mathbb{N}$, $\varepsilon>0$ and $i\in\mathbb{Z}$ be fixed. A subset $E\subseteq M_{i}$ is called   $(n,\varepsilon)$-\textit{separated with  respect to} \textbf{\textit{f}} if given $x,y\in E$, with $x\neq y$, we have $\max_{0\leq j< n}\textbf{d}( \textbf{\textit{f}}_{i} ^{j}(x), \textbf{\textit{f}}_{i} ^{j}(y))>\varepsilon$, i. e., if for all $x\in E$, the set \( \bigcap _{k=0}^{n-1}(\textbf{\textit{f}}_{i}^{\,k})^{-1}(\overline{B(\textbf{\textit{f}}_{i}^{\,k}(x),\varepsilon)})\)
  contains no other point of $E$. 
  \end{defin}
Denote by $s[n,i](\varepsilon,\textbf{\textit{f}})$   the largest cardinality of any  $(n,\varepsilon)$-separated subset of $M_{i}$ with respect to $\textbf{\textit{f}}$.
Set \( s[i](\varepsilon,\textbf{\textit{f}})=\underset{n\rightarrow +\infty}\limsup \frac{1}{n}\log s[n,i](\varepsilon,\textbf{\textit{f}}).\) 

  \begin{propo}\label{emparedado} Given $\varepsilon>0$ and $i\in\mathbb{Z}$ we have:
 \begin{enumerate}
 \item $r[n,i](\varepsilon,\textbf{f})\leq s[n,i](\varepsilon,\textbf{f})\leq r[n,i](\varepsilon/2,\textbf{f})$, for all $n>0$. 
 \item $r[i](\varepsilon,\textbf{f})\leq s[i](\varepsilon,\textbf{f})\leq r[i](\varepsilon/2,\textbf{f})$, for all  $n>0$. 
 \end{enumerate}
 \end{propo}

From Proposition \ref{emparedado} we have  \(\textbf{H}_{i}(\textbf{\textit{f}})=\lim_{\varepsilon\rightarrow 0}s[i](\varepsilon,\textbf{\textit{f}})\) for all $i\in\mathbb{Z}$. 
 Consequently,   $\textbf{H}_{i}(\textbf{\textit{f}})$ can be defined using either span or separated subsets.
  
\medskip

Notice that if  $ \textit{\textbf{f}} $ is a constant family associated to a  homeomorphism $\phi:X\rightarrow X,$ then it is clear that 
 \begin{equation}\label{igualdadeentropiassd}
   \textbf{H}_{i}(\textbf{\textit{f}})=h(\phi),\quad\text{ for all  }i\in\mathbb{Z}.
\end{equation}
Therefore,   \textbf{H} generalizes the notion of  topological entropy for single homeomorphisms.

\medskip

Some estimations of the topological entropy for non-autonomous dynamical
systems can be found in   \cite{Shao}, \cite{Zhang} and  \cite{Zhu}. 
 
\section{Some Properties of the Entropy}
In this section we will see some properties of the topological entropy.  Some of them are analogous to the well-known properties of entropy for single maps. For singular maps, the topological entropy is invariant by  topological conjugacies. The main result of this section is to prove the analogous result for non-autonomous dynamical systems, that is, the entropy is invariant by uniformly conjugacies between sequences (see Theorem \ref{invarianteentropia}).    This result will be fundamental to show the continuity of the entropy in Section 6 (see Theorem \ref{teoprinc}).   
 
\medskip

 As we had mentioned,  the notions of entropy for families of homeomorphisms, considering either open covers or separated subsets,   coincide. This fact can be proved analoguosly  as in the case of single homeomorphisms (see \cite{Walters}, Chapter 7, Section 2): 

\begin{propo}\label{igualdadeentropias}  For each $i\in\mathbb{Z}$ we have $\mathcal{H}_{i}(\textbf{\textit{f}})=\textbf{H}_{i}(\textbf{\textit{f}}).$ 
\end{propo}

The topological entropy $h(\phi)$ of a single homeomorphism  
$\phi:X\rightarrow X$  satisfies $h(\phi^{n})=|n|h(\phi),$ for $n\in \mathbb{Z}.$ 
For families we have: 
\begin{propo}\label{gateringn} Suppose  $ \textbf{\textit{f}}=(f_{i})_{i\in\mathbb{Z}}$ is an equicontinuous sequence. Fix    $n\geq1.$ Let      $(\widetilde{\textbf{M}},\tilde{\textbf{f}})$ be the gathering obtained of $(\textbf{M},\textbf{f})$  by the sequence $(ni)_{i\in\mathbb{Z}}$,  that is,   $\widetilde{M}_{i}=M_{ni}$ and $\tilde{f}_{i}=f_{n(i+1)-1}\circ \cdots \circ f_{ni};$  
\[\begin{CD}
\cdots M_{n(i-1)}@>{\tilde{f}_{i-1}=f_{ni-1}\circ \cdots  \circ f_{n(i-1)}}>> M_{ni}@>{\tilde{f}_{i}=f_{n(i+1)-1}\circ \cdots  \circ f_{ni}}>>M_{n(i+1)}   \cdots
\end{CD} 
\]
Thus, for each $i\in\mathbb{Z}$ we have  \(\textbf{H}_{i}(\tilde{\textbf{\textit{f}}} )=n\textbf{H}_{in}(\textbf{\textit{f}}).\)
\end{propo}
\begin{proof}For $i\in\mathbb{Z}$, $x,y\in M_{ni}$ and $m>0$, we have 
 \[\max_{0\leq k< m}\textbf{d}(\tilde{\textbf{\textit{f}}}_{i} ^{k}(x),\tilde{\textbf{\textit{f}}}_{i} ^{k}(y))=\max_{0\leq k< m}\textbf{d}( \textbf{\textit{f}}_{ni} ^{nk}(x), \textbf{\textit{f}}_{ni}^{nk}(y))\leq \max_{0\leq j< nm}\textbf{d}( \textbf{\textit{f}}_{ni} ^{j}(x), \textbf{\textit{f}}_{ni} ^{j}(y)).\]
 This fact proves that, for all $\varepsilon>0$,  each  $(nm,\varepsilon)$-span subset $K$ of $M_{ni}$ with respect to  \textbf{\textit{f}} is a $(m,\varepsilon)$-span subset of $M_{ni}$ with respect to $\tilde{\textbf{\textit{f}}}$.  Consequently, we obtain  $r[m,ni](\varepsilon,\tilde{\textbf{\textit{f}}})\leq r[nm,ni](\varepsilon,\textbf{\textit{f}})$.  Hence,   \( \textbf{H}_{i}(\tilde{\textbf{\textit{f}}})\leq  n\textbf{H}_{in}(\textbf{\textit{f}}).\)
 
 On the other hand, since \textbf{\textit{f}} is equicontinuous, we can prove that $(f_{ni})_{i\in \mathbb{Z}}$, $ (\textbf{\textit{f}}_{ni}^{2})_{i\in \mathbb{Z}}$,\dots, $(\textbf{\textit{f}}_{ni}^{n-1})_{i\in \mathbb{Z}}$ is a   collection of equicontinuous families. Consequently, given $\varepsilon>0$, there exists $\delta>0$ such that 
 $$\max_{\substack{1\leq k<n \\j\in\mathbb{Z}}} \{\textbf{d}(f_{nj+k}  \cdots   f_{nj+1} f_{nj}(x),f_{nj+k} \cdots f_{nj+1} f_{nj}(y)): x,y\in M_{nj}, \textbf{d}(x,y)<\delta\}<\varepsilon.$$ 
Now, if $K$ is a $(m,\delta)$-span of $M_{ni}$ with  respect to $\tilde{\textbf{\textit{f}}}$, then, for all $x\in M_{ni}$, there exists $y\in K$ such that
   \begin{equation*}  \max \{\textbf{d}(x,y),\textbf{d}( \textbf{\textit{f}}_{ni} ^{n} (x), \textbf{\textit{f}}_{ni}  ^{n}(y)), ...,\textbf{d}(\textbf{\textit{f}}_{ni} ^{(m-1)n}(x), \textbf{\textit{f}}_{ni} ^{(m-1)n}(y)) \}   
<\delta. \end{equation*}
Thus, 
  \begin{align*} \max_{0\leq k <n}&  \{ \textbf{d}(\textbf{\textit{f}}_{ni} ^{k}(x), \textbf{\textit{f}}_{ni} ^{k}(y)) \}<\varepsilon, 
  \quad \max_{0\leq k <n} \{ \textbf{d}(\textbf{\textit{f}}_{n(i+1)} ^{k}\circ\textbf{\textit{f}}_{ni} ^{n}(x),\textbf{\textit{f}}_{n(i+1)} ^{k}\circ\textbf{\textit{f}}_{ni} ^{n}(y))\}<\varepsilon, \dots,\\  
 &\max_{0\leq k <n}  \{ \textbf{d}(\textbf{\textit{f}}_{n(i+m-1)} ^{k}\circ\textbf{\textit{f}}_{ni} ^{(m-1)n}(x),\textbf{\textit{f}}_{n(i+m-1)} ^{k}\circ\textbf{\textit{f}}_{ni} ^{(m-1)n}(y))\}<\varepsilon.
 \end{align*}
 Consequently, we have 
 \begin{align*} \max_{0\leq k <n}&  \{ \textbf{d}(\textbf{\textit{f}}_{ni} ^{k}(x), \textbf{\textit{f}}_{ni} ^{k}(y))\}<\varepsilon,\quad
    \max_{0\leq k <n}  \{\textbf{d}(\textbf{\textit{f}}_{ni} ^{n+k}(x), \textbf{\textit{f}}_{ni} ^{n+k}(y)) \}<\varepsilon,\dots,\\  
 &\max_{0\leq k <n}  \{ \textbf{d}(\textbf{\textit{f}}_{ni} ^{(m-1)n+k}(x), \textbf{\textit{f}}_{ni} ^{(m-1)n+k}(y)) \}<\varepsilon.
 \end{align*}
Therefore, \[\max\{\textbf{d}(\textbf{\textit{f}}_{ni} ^{k}(x), \textbf{\textit{f}}_{ni} ^{k}(y)):k=0,...,mn-1 \}<\varepsilon,\] that is, $K$ is a $(mn,\varepsilon)$-span of $M_{ni}$ with respect to $\textbf{\textit{f}}$. Hence, we have $r[m,ni](\varepsilon,\tilde{\textbf{\textit{f}}})\geq r[nm,ni](\varepsilon,\textbf{\textit{f}})$ and, therefore,  \( \textbf{H}_{i}(\tilde{\textbf{\textit{f}}})\geq  n\textbf{H}_{in}(\textbf{\textit{f}}), \)  which proves the proposition. 
\end{proof}

From the proof of Proposition \ref{gateringn},  we have always the inequality  \[\textbf{H}_{i}(\tilde{\textbf{\textit{f}}} )\leq n\textbf{H}_{in}(\textbf{\textit{f}}).\]

\begin{propo}Suppose $ \textbf{\textit{f}}=(f_{i})_{i\in\mathbb{Z}}$  is a sequence consisting of  isometries, that is, $f_{i}:M_{i}\rightarrow M_{i+1}$ is an isometry for all $i.$ Thus $\textbf{H}_{i}( \textbf{\textit{f}})=0, $
for all $i\in\mathbb{Z}.$ 
\end{propo}
\begin{proof} If follows directly from  Definition \ref{entropiametrica}.
\end{proof}
 In the following  theorem we will see that the entropy for non-autonomous dynamical systems is invariant for  uniform  conjugacies. This result generalizes the fact that the topological entropy of homeomorphisms defined on compact metric spaces is invariant by topological conjugacies. 
 \begin{teo}\label{invarianteentropia} If $(\textbf{M},\textbf{\textit{f}})$ and $(\textbf{N},\textbf{\textit{g}})$ are  uniformly    conjugate, then $ \textbf{H}_{i}( \textbf{\textit{f}})= \textbf{H}_{i}(\textbf{\textit{g}})$ for all $i\in\mathbb{Z}.$ 
\end{teo}
\begin{proof}Fix $i\in\mathbb{Z}$. Let $\textbf{\textit{h}}=(h_{i})_{i\in\mathbb{Z}}$ be a uniform conjugacy between
\textbf{\textit{f}} and \textbf{\textit{g}}. Since \textbf{\textit{h}} is equicontinuous, given $\varepsilon>0$ there exists $\delta>0$ such that, for all $j\geq i $, if $x,y\in M_{j}$ and $\textbf{d}(x,y)<\delta$, then $\textbf{d}(h_{j}(x),h_{j}(y))<\varepsilon$.   Let $K$ be a $(m,\delta)$-span of $M_{i}$ with respect to $ \textbf{\textit{f}}$. Thus, for all $x\in M_{i}$ there exists $y\in K$ such that
 $\max_{0\leq j< m}\textbf{d}(  \textbf{\textit{f}}_{i} ^{j}(x),  \textbf{\textit{f}}_{i} ^{j}(y))<\delta$.  Consequently,   if $0\leq j< m$,  \[\varepsilon>\max_{0\leq j< m}\textbf{d}(h_{i+j}\circ\textbf{\textit{f}}_{i} ^{j}(x),  h_{i+j}\circ\textbf{\textit{f}}_{i} ^{j}(y))=\max_{0\leq j< m}\textbf{d}( \textbf{\textit{g}}_{i} ^{j}\circ h_{i}(x), \textbf{\textit{g}}_{i} ^{j} \circ h_{i} (y)).\] This fact proves that $r[m,i](\varepsilon, \textbf{\textit{f}})\geq r[m,i](\delta,\textbf{\textit{g}})$. Hence, 
 $\textbf{H}_{i}( \textbf{\textit{f}})\geq  \textbf{H}_{i}(\textbf{\textit{g}})$. Since  $\textbf{\textit{h}}^{-1}$ is equicontinuous, analoguosly we can prove  that $\textbf{H}_{i}( \textbf{\textit{f}})\leq  \textbf{H}_{i}(\textbf{\textit{g}})$.
\end{proof}

 It follows from the proof of the above theorem that if $(f_{i})_{i\geq i_{0}}$ and  $(g_{i})_{i\geq i_{0}}$  are  uniformly conjugate   (see Lemma \ref{lemapositive}) then  $\textbf{H}_{i_{0}}( \textbf{\textit{f}})=  \textbf{H}_{i_{0}}(\textbf{\textit{g}})$.  Furthermore,      the entropy for homeomorphisms depends only on the   future:
  
\begin{cor}\label{entropiigual} Suppose that there exists $i_{0}\in\mathbb{Z}$ such that  $f_{j}=g_{j}$ for all $j\geq i_{0}$.  Then for all $i\in\mathbb{Z}$ we have $\textbf{H}_{i}( \textbf{\textit{f}})=  \textbf{H}_{i}(\textbf{\textit{g}})$.
\end{cor}
\begin{proof}It is clear that $(f_{j})_{j\geq i_{0}}$ and  $(g_{j})_{j\geq i_{0}}$ are uniformly conjugate (take $h_{j}=Id$ for each $j\geq i_{0}$).  It follows from Lemma \ref{lemapositive} that, for any $i\in\mathbb{Z}$, $(f_{j})_{j\geq i}$ and  $(g_{j})_{j\geq i}$ are uniformly conjugate. By the proof of the Theorem \ref{invarianteentropia}   we have $\textbf{H}_{i}( \textbf{\textit{f}})=  \textbf{H}_{i}(\textbf{\textit{g}})$ for all $i\in\mathbb{Z}$.
\end{proof}

\begin{cor}\label{corolarioigualdad} For all $i,j\in\mathbb{Z}$ we have $\textbf{H}_{i}( \textbf{\textit{f}})=  \textbf{H}_{j}(\textbf{\textit{f}})$.
\end{cor}
\begin{proof} It is sufficient to prove that $\textbf{H}_{i}( \textbf{\textit{f}})=  \textbf{H}_{i+1}(\textbf{\textit{f}})$ for all $i\in\mathbb{Z}$. Fix $i\in\mathbb{Z}$.  Take the family $\textbf{\textit{g}}=(g_{j})_{j\in \mathbb{Z}}$, where $g_{j}=I: M_{j}\rightarrow M_{j+1}$ for each $j\leq i$,  the identity on $M$, and $g_{j}=f_{j}$ for $j>i$.   Thus $\textbf{H}_{i}( \textbf{\textit{f}})=  \textbf{H}_{i}(\textbf{\textit{g}})$. For each $x,y\in M_{i}$ and $n\geq 2$  we have $$\max_{0\leq j< n}\textbf{d}( \textbf{\textit{g}}_{i} ^{j}(x), \textbf{\textit{g}}_{i} ^{j}(y))= \max_{0\leq j< n-1}\textbf{d}( \textbf{\textit{g}}_{i+1} ^{j}(x), \textbf{\textit{g}}_{i+1} ^{j}(y)).$$
Using this fact we can prove that  $\textbf{H}_{i}( \textbf{\textit{g}})=  \textbf{H}_{i+1}(\textbf{\textit{g}})$. Consequently, we have that  $\textbf{H}_{i}( \textbf{\textit{f}})=  \textbf{H}_{i+1}(\textbf{\textit{f}})$.
\end{proof}

\begin{obs}\label{remarkunmap} We can consider the sequence $\textbf{\textit{f}}=(f_{i})_{i\in\mathbb{Z}}$ as a homeomorphism $\textbf{\textit{f}}:\textbf{M}\rightarrow \textbf{M}$ and then calculate the topological entropy of $h(\textbf{\textit{f}})$ via spanning or separated sets of \textbf{M}.  It is not difficult to prove that $h(\textbf{\textit{f}})= \textbf{H}_{i}( \textbf{\textit{f}})$ for any $i$. Hence, from now on we will   omit the index $i$ of $\textbf{H}_{i}$ and we will   consider the entropy of a non-autonomous dynamical system as a single number, as a consequence of Corollary \ref{corolarioigualdad}.\end{obs}

 Remember that I am  fixing one metric $d$ on $M$ and then  considering the metric on the total space as in \eqref{metricatotal1}. If we consider another metric $\tilde{d}$ uniformly equivalent to $d$ on  $M$, then the identity  \begin{align*}I: (\textbf{M},\textbf{d})&\rightarrow (\textbf{M},\tilde{\textbf{d}}) \\
p&\mapsto p
\end{align*} 
is  a uniformly continuous map. It follows from Theorem 7.4 in \cite{Walters} that  the topological entropy of  $\textbf{\textit{f}}$  considering the metric $\tilde{d}$ on $M$ coincides with the topological entropy of  $\textbf{\textit{f}}$ considering $d$ on $M$. If follows that the entropy for a non-autonomous dynamical system on \textbf{M} does not depend on equivalent metrics on $M$.

\medskip
 
We can define the inverse of $\textbf{\textit{f}}$ as   $ \textbf{\textit{f}}^{-1}=(g _{i})_{i\in\mathbb{Z}}$, where $g _{i}:=f_{i}^{-1}:M_{i+1}\rightarrow M_{i}$ for each $i$. 
In this case, \((\textbf{\textit{f}}^{-1})_{ i} ^{0}:= I_{i+1}:M_{i+1}\rightarrow M_{i+1}\)  and \(
   (\textbf{\textit{f}}^{-1})_{ i} ^{ n}:= g_{i-n+1}\circ \cdots\circ g_{i}:M_{i+1}\rightarrow M_{i-n+1}\)  for \(n>0.\) 
 In the case of a single homeomorphism $\phi:X\rightarrow X$, we have $h(\phi)=h(\phi^{-1})$   (see \cite{Walters}, Theorem 7.3).  The following example proves that, in general,  we could have $\textbf{H}(\textbf{\textit{f}})\neq \textbf{H} (\textbf{\textit{f}}^{-1})$. 
 \begin{ej}\label{ejemplootraentropia} Let $I:M\rightarrow M$ be the identity on $M$  and $\phi:M\rightarrow M$ be  a homeomorphism on $M$ with non-zero topological entropy. Let $f_{i}:M_{i}\rightarrow M_{i+1}$    be the diffeomorphisms  defined  as 
 \(f_{i} =I\) for \(i\geq0\)  and \(f_{i} =
\phi \) for \(i<0\)
and take  $\textbf{\textit{f}}=(f_{i})_{i\in\mathbb{Z}}$.  
From Corollary \ref{entropiigual} we have  $\textbf{H}(\textbf{\textit{f}} )=h(I)=0$ and  $\textbf{H}(\textbf{\textit{f}}^{-1})=h(\phi)\neq0$, for each $i\in\mathbb{Z}$. 
\end{ej}

\begin{obs}\label{obsentropiainversa} As a consequence of Example \ref{ejemplootraentropia},  we   can also consider  the  entropy $ \textbf{H}(\textbf{\textit{f}}^{-1})$, which we denote by $\textbf{H}^{(-1)}(\textbf{\textit{f}})$. All the above results for $\textbf{H} $  have analogous versions for $\textbf{H}^{(-1)}$.
\end{obs}

There are dynamical systems defined on a compact metric space that are not topologically conjugate but they have the same topological entropy. Now,   from Theorem \ref{invarianteentropia} we have that  two constant families associated to homeomorphisms with different topological entropies can not be uniformly conjugate.  
 On the other hand, in  \cite{Jeo3} is proved that  there are constant families, associated to homeomorphisms with the same topological entropy, that can be uniformly topologically conjugate. One natural question  that arise from this notion of entropy are as follows:
  Let $(\textbf{M},\textbf{\textit{f}})$ and $(\textbf{M},\textbf{\textit{g}})$ be constant families. If $\textbf{H}( \textbf{\textit{f}})=  \textbf{H} (\textbf{\textit{g}})$    then \textbf{\textit{f}} and \textbf{\textit{g}}   are always uniformly  conjugate? The answer   is negative, as shows the following example:
 \begin{ej} Let $\psi:X\rightarrow X$ be a homeomorphism on a metric space $X$ with metric $\rho.$ For  $x\in X$,  set 
 \[W^{s}(x,\psi)=\{y\in X: \rho(\psi^{n}(x), \psi^{n}(y))\rightarrow 0 \text{ as }n \rightarrow +\infty\} .\]
 This set is called the \text{stable set for} $\psi$  at $x.$ In  \cite{Jeo3}  is proved that, if $\textbf{\textit{h}}=(h_{i})_{i\in\mathbb{Z}}$ is a uniform conjugacy between  $(\textbf{M},  \textbf{\textit{f}})$ and $(\textbf{N},  \textbf{\textit{g}})$, then, for each $x\in M_{i}$, we have \[h_{i}(W^{s}(x,\textbf{\textit{f}})) = W^{s} (h_{i}(x),\textbf{\textit{g}})\quad\text{ and }\quad h_{i}(W^{u}(x,\textbf{\textit{f}})) = W^{u} (h_{i}(x),\textbf{\textit{g}}).\] 
 
   Let $M$ be $\mathbb{S}^{1}$, $p_{N}$ be the north pole and $p_{S}$ be the south pole of $\mathbb{S}^{1}$. Suppose that $\phi:M\rightarrow M$ is a homeomorphism  with stable set $ W^{s}(p_{N},\phi)=M\setminus \{p_{S}\}$. Let $\textbf{\textit{f}}$ and $\textbf{\textit{g}}$ be the   constant families associated to $\phi$ and to the identity on $M$, respectively. Then $\textbf{H}( \textbf{\textit{f}})=  \textbf{H}(\textbf{\textit{g}})=0$ for all $i\in\mathbb{Z}$, because all the homeomorphisms on the circle has zero entropy (see \eqref{igualdadeentropiassd}).  On the other hand, we have  $W^{s}((p_{N},0),\textbf{\textit{f}}) = [M\setminus \{p_{S}\}]\times \{0\}$  and $W^{s}((p_{N},0),\textbf{\textit{g}}) = \{(p_{N}, 0)\}.$ Since the uniform conjugacies preserve the stable sets, we have that  $\textbf{\textit{f}}$ and $\textbf{\textit{g}}$ can not be uniformly conjugate.  
\end{ej}

\section{On the Continuity of the Entropy}

 Finally we will  see that the entropy is continuous considering the strong topology on $\text{F}^{r}(\textbf{M})$, for  $r\geq  1$. More specifically, the entropy  is locally constant, that is,   each $(f_{i})_{i\in\mathbb{Z}}\in \text{F}^{r}(\textbf{M})$  has a strong basic neighborhood in which the entropy is constant.   In contrast,  on the continuity of $\textbf{H}:(\text{F}^{r}(\textbf{M}), \tau_{prod})\rightarrow \mathbb{R}\cup \{+\infty\}$, we have: 

\begin{propo} Suppose that $\textbf{H}(\text{F}^{r}(\textbf{M}))$ has two or more elements. Then $\textbf{H}:(\text{F}^{r}(\textbf{M}), \tau_{prod})\rightarrow \mathbb{R}\cup \{+\infty\}$ is discontinuous at any $\textbf{\textit{f}}\in \text{CF}^{r}(\textbf{M})$. 
\end{propo}
\begin{proof} Let $\textbf{\textit{f}}=(f_{i})_{i\in\mathbb{Z}}\in \text{F}^{r}(\textbf{M}).$ Since $\textbf{H}(\text{F}^{r}(\textbf{M}))$ has two or more elements, there exists $\textbf{\textit{g}}=(g_{i})_{i\in\mathbb{Z}}\in \text{F}^{r}(\textbf{M}) $ such that 
$\textbf{H}(\textbf{\textit{g}})\neq \textbf{H}(\textbf{\textit{f}}).$ Let $\mathcal{V}\in \tau_{prod}$ an open neighborhood of $\textbf{\textit{f}}$. For some $k\in \mathbb{N}$, the family $\textbf{\textit{h}}=(h_{i})_{i\in\mathbb{Z}}$, defined by
\begin{equation*}h_{i}=  
\begin{cases}
  f_{i}  & \mbox{if }-k\leq i\leq k \\
  g_{i} & \mbox{if }i>k \mbox{ or }i<-k, \\ 
        \end{cases} 
\end{equation*}
belongs to $\mathcal{V}$, by definition of $\tau_{prod}$. It is follow from Corollary \ref{entropiigual} that \[\textbf{H}(\textbf{\textit{h}})= \textbf{H}(\textbf{\textit{g}}),\]
which proves the proposition, since $(\text{F}^{r}(\textbf{M}), \tau_{prod})$ a metric space.  
\end{proof}

Set \[\text{CF}^{r}(\textbf{M})=\{\textbf{\textit{f}}\in  \text{F}^{r}(\textbf{M}): \textbf{\textit{f}}\text{ is a constant family}\},\]
  $\tilde{\tau}_{str}=\tau_{str}|_{\text{CF}^{r}(\textbf{M})}$  and $\tilde{\tau}_{prod}=\tau_{prod}|_{\text{CF}^{r}(\textbf{M})}.$ 
  
\begin{propo}\label{igualdadtopologias}  $\tilde{\tau}_{str}=\mathcal{P}(\text{CF}^{r}(\textbf{M}))=\{A:A\subseteq \text{CF}^{r}(\textbf{M})\}.$   
\end{propo}
  \begin{proof}It is sufficient to prove that each $\{ (f_{i})_{i\in\mathbb{Z}}\}$, with $ (f_{i})_{i\in\mathbb{Z}}\in \text{CF}^{r}(\textbf{M})$, is open  in $\text{CF}^{r}(\textbf{M})$. Let $(\varepsilon_{i})_{i\in\mathbb{Z}}$ be a sequence of positive numbers with $\varepsilon_{i}\rightarrow 0$ as $|i|\rightarrow \pm \infty$. Consider the strong basic neighborhood $B^{r}((f_{i})_{i\in\mathbb{Z}},(\varepsilon_{i})_{i\in\mathbb{Z}})\subseteq \text{F}^{r}(\textbf{M})$ of $(f_{i})_{i\in\mathbb{Z}}$.  Notice that \[\{ (f_{i})_{i\in\mathbb{Z}}\}= B^{r}((f_{i})_{i\in\mathbb{Z}},(\varepsilon_{i})_{i\in\mathbb{Z}})\cap \text{CF}^{r}(\textbf{M}).\]
  Consequently,  $\{ (f_{i})_{i\in\mathbb{Z}}\}$  is open  in $(\text{CF}^{r}(\textbf{M}),\tilde{\tau}_{str}.$
  \end{proof}
   The application 
\begin{align*} \pi_{0} : (\text{F}^{r}(\textbf{M}),\tau)&\rightarrow (\text{Diff}^{r}(M_{0},M_{1}),d^{r})\\
(f_{i})_{i\in\mathbb{Z}}& \mapsto f_{0}
\end{align*}
is continuous for $\tau\in \{\tau_{str},\tau_{prod} \}$. 
Hence, the restriction
 \begin{align*}\tilde{\pi}_{0}= \pi_{0}|_{\text{CF}^{r}(\textbf{M})} : (\text{CF}^{r}(\textbf{M}),\tilde{\tau}) \rightarrow (\text{Diff}^{r}(M_{0},M_{1}),d^{r})
\end{align*}
is continuous for 
$\tilde{\tau}\in \{\tilde{\tau}_{str},\tilde{\tau}_{prod} \}$. We can identify $(\text{Diff}^{r}(M_{0},M_{1}),d^{r}) $ with the space $(\text{Diff}^{r}(M),d^{r})$, the space consisting of diffeomorphisms on $M$ endowed with the $C^{r}$-metric obtained from the metric $d$ on $M$. 
From now on we will make use of this identification.  For a $C^{r}$-diffeomorphism $\phi:M\rightarrow M$, we denote the constant family associated to $\phi$ by $\textbf{\textit{f}}_{\phi}$. Notice that $\tilde{\pi}_{0}$ is invertible, in fact, 
\begin{align*}\tilde{\pi}_{0}^{-1} : (\text{Diff}^{r}(M),d^{r}) &\rightarrow  (\text{CF}^{r}(\textbf{M}),\tilde{\tau}) \\
\phi& \mapsto \textbf{\textit{f}}_{\phi}.
\end{align*}

Clearly, if $\tilde{\tau}=\tilde{\tau}_{str}$, then 
$\tilde{\pi}_{0}^{-1}$ is not continuous (see Proposition \ref{igualdadtopologias}). On the other hand, we have:

\begin{propo} If $\tilde{\tau}=\tilde{\tau}_{prod}$, then $\tilde{\pi}_{0}^{-1}$ is continuous.  
\end{propo}
\begin{proof} All the open subsets of  $(\text{CF}^{r}(\textbf{M}),\tilde{\tau}_{prod}) $ are union of sets with  the form \[  \mathcal{U}=\left(\prod_{i<-j}   \text{Diff}^{r}(M_{i},M_{i+1})\times \prod_{i=-j}^{j}[U_{i} ]\times \prod_{i>j}   \text{Diff}^{r}(M_{i},M_{i+1})\right)\cap \text{CF}^{r}(\textbf{M}), \]
where $U_{i}$ is an open subset of $\text{Diff}^{r}(M_{i},M_{i+1})$, for $-j\leq i\leq j.$ Notice that \[(\tilde{\pi}_{0}^{-1})^{-1}(\mathcal{U})=\tilde{\pi}_{0}(\mathcal{U}) = \bigcap _{i=-j}^{j} U_{i},\]
which is an open subset of $\text{Diff}^{r}(M)$. Thus, $\tilde{\pi}_{0}^{-1}$ is continuous.
\end{proof}

Consequently, we have:

\begin{propo}\label{propodeee}  $\textbf{H} :(\text{CF}^{r}(\textbf{M}),\tilde{\tau}_{prod})\rightarrow \mathbb{R} ,$   is continuous if, and only if, 
  $h:(\text{Diff}^{r}(M),d^{r})\rightarrow \mathbb{R}  $    is continuous.
\end{propo}
\begin{proof}It is clear, because  $\textbf{H}  =h\circ \tilde{\pi}_{0} $ and $\tilde{\pi}_{0} $ is a homeomorphism. 
\end{proof}

\begin{obs}Proposition \ref{propodeee} could be  a useful tool to show the continuity of the topological entropy  at some  $C^{r}$-diffeomorphisms: to show that  $h$ is continuous at $\phi\in \text{Diff}^{\,r}(M)$,  we could try to prove that  $\textbf{H}|_{\text{CF}^{r}(\textbf{M})}$ is continuous at $\textbf{\textit{f}}_{\phi}$. In order to prove this  fact, we have to find an open  neighborhood $\mathcal{U}\subseteq \text{Diff}^{r}(M)$ of $\phi$, such that each constant family associated to any diffeomorphism in  $\mathcal{U}$ is uniformly conjugate to $\textbf{\textit{f}}_{\phi}$. Thus, by   Theorem  \ref{invarianteentropia} and \eqref{igualdadeentropiassd}, we had that $$h( \psi) =\textbf{H}(\textbf{\textit{f}}_{\psi})= \textbf{H}(\textbf{\textit{f}}_{\phi})=h( \phi)\quad\text{ for any }\psi\in \mathcal{U}.  $$

In  \cite{Jeo3} we will prove that there exist diffeomorphisms $\phi$ and $\psi$ which are not topologically conjugate,  however $\textbf{\textit{f}}_{\phi}$ and $\textbf{\textit{f}}_{\psi}$ could be uniformly   conjugate.
\end{obs}

Finally, we will prove the continuity of $\textbf{H}:(\text{F}^{r}(\textbf{M}), \tau_{str})\rightarrow \mathbb{R}\cup\{+\infty\}$ for any $ r\geq 1$.  It is sufficient to prove the case when    $ r=1$. 

\medskip

Remember we are supposing that $M$ is a compact Riemannian manifold   with Riemannian norm $\Vert\cdot\Vert$, which induces a metric $d$ on $M$, and then  we   consider  the metric $\textbf{d}$ on $\textbf{M}$ as in \eqref{metricatotal1}. Let $\varrho>0$    be such that, for each $x\in M$, the exponential application   \[\text{exp}_{x}:B(0_{x},\varrho) \rightarrow B(x,\varrho) \] is a diffeomorphism and \( \Vert v\Vert =d(\text{exp}_{x}(v),x),\)  for all  \(v\in B(0_{x},\varrho),
\)
 that is, $\varrho$ is the  \textit{injectivity radius} of $M$.\footnote{Here, $0_{x}$ is the zero vector in $T_{x}M$, the tangent space of $M$ at $x$.} We will suppose that $\varrho<1/2$.

\medskip

We will fix $\textbf{\textit{f}}=(f_{i})_{i\in\mathbb{Z}}\in \text{F}^{1}(\textbf{M}).$   
For $\delta>0$ and $r=0,1$, set 
\[ D^{r}(I_{i},\delta)=\{h\in \text{Hom} (M_{i},M_{i}): h \text{ is a } C^{r}\text{-diffeomorphism and }d^{r}(h ,I_{i})\leq \delta\}\]
\[\text{and }D^{1}(f_{i},\delta)=\{g\in \text{Diff} ^{1}(M_{i},M_{i+1}):  d^{1}(g ,f_{i})\leq \delta\}.\]

  The closure of $D^{1}(I_{i},\delta)$ on $D^{0}(I_{i},\delta)$ will be denoted by   $\overline{D^{1}(I_{i},\delta)}.$ 

 \begin{lem}\label{escolhari} There exist two sequences $(r_{i})_{i\geq 0}$ and $(\delta_{i})_{i\geq 0}$, with $r_{i}\rightarrow 0$ as $i\rightarrow +\infty$, such that, for each $ g\in  D^{1}(f_{i},\delta_{i})$, the map 
 \begin{align*} \tilde{G}_{i+1}:D^{r}(I_{i+1},r_{i+1})&\rightarrow  D^{r}(I_{i},r_{i} )\\
 h&\mapsto g ^{-1}h f_{i} 
 \end{align*}
 is well-defined for each $i\geq1$. 
 \end{lem}
 \begin{proof}
  Notice that if  $g\in \text{Diff}^{1}(M_{i},M_{i+1})$ and $h\in \text{Diff}^{1}(M_{i+1},M_{i+1}),$ we have
\[d^{1}(g ^{-1} h f_{i} , I_{i})\leq d^{1}(g^{-1}   h f_{i}  , g ^{-1}   f_{i}  )+  d^{1}(g^{-1}    f_{i}  , I_{i})\quad\text{ for }i\geq0 .\]
If $h$ is $C^{1}$-close to $I_{i+1}$, then $g ^{-1} h f_{i}$ is $C^{1}$-close to $g ^{-1}  f_{i}$ and if $g$ is $C^{1}$-close to $f_{i}$, then $g^{-1}    f_{i} $ is $C^{1}$-close to $I_{i}$.   Fix $r_{0}\in(0,\varrho/4)$. There exist $r_{1}\in (0,r_{0}/2)$ and $\delta_{0}>0$ such that, if $h\in D^{1}(I_{1},r_{1})$ and $g \in D^{1}(f_{0},\delta_{0})$,  then $g^{-1}hf_{0}\in D^{1}(I_{0},r_{0})$.  Take $r_{2}\in (0,r_{1}/2)$ and $\delta_{1}>0$ such that, if $h\in D^{1}(I_{2},r_{2})$ and $g \in D^{1}(f_{1},\delta_{1})$,  then $g^{-1}hf_{1}\in D^{1}(I_{1},r_{1})$. Hence, inductively, we can build two sequences $(r_{i})_{i\geq 0}$ and $(\delta_{i})_{i\geq 0}$, with $r_{i}\in (0,r_{i-1}/2)$ for each $i\geq 1$, such that $g^{-1}hf_{i}\in D^{1}(I_{i},r_{i})$, which proves the lemma.
\end{proof}

Analogously,  we can find a  sequence of positive numbers $(\delta_{i})_{i\leq 0}$ and $(r_{i})_{i\leq 0}$, with $r_{i}\rightarrow 0$ as $i\rightarrow -\infty$,   such that for each $g\in  D^{1}(f_{i-1}, \delta_{i-1})$, the map 
 \begin{align*} \hat{G}_{i-1}: D^{r}(I_{i-1},r_{i-1})&\rightarrow  D^{r}(I_{i},r_{i})\\
 h&\mapsto f_{i-1} h g ^{-1} 
 \end{align*}
 is well-defined for each $i\leq 0$.

\begin{lem}\label{existenciapuntofijo} There exist two sequences   $\tilde{\textbf{\textit{h}}}=(\tilde{h}_{i})_{i\geq0}\in  \prod_{i\geq0}D^{0}(I_{i},r_{i})$ and $\hat{\textbf{\textit{h}}}=(\hat{h}_{i})_{i\leq0}\in  \prod_{i\leq0}D^{0}(I_{i},r_{i})  $ such that 
\[\tilde{G}_{i+1}\tilde{h}_{i+1}=\tilde{h}_{i}\text{ for all }i\geq 0\quad \text{and}\quad \hat{G}_{i-1}\hat{h}_{i-1}=\hat{h}_{i}\text{ for all }i\leq 0 .\] 
\end{lem}
\begin{proof}
For each $i> 0$, let $h^{i}=\tilde{G}_{1}\circ\cdots\circ \tilde{G}_{i}(I_{i}).$ 
It follows from Lemma \ref{escolhari}   that $h^{i}$ belongs to $D^{1}(I_{0},r_{0})$.  
 Consequently, the sequence $(h^{i})_{i\geq 0}$ is equicontinuous, because each $h^{i}$ is  $C^{1}$ and the sequence has    uniformly bounded derivative.    Hence, there  exist   a subsequence $i_{m}\rightarrow \infty$ and $\tilde{h}_{0}\in D^{0}(I_{0},r_{0})$ such that $h^{i_{m}}\rightarrow \tilde{h}_{0}$ as  $m\rightarrow  \infty$. Notice that   $G_{1}$ is invertible and both $G_{1}$ and ${G _{1}} ^{-1}$ are continuous. Consequently, 
 \[ G_{1}(\overline{D^{1}(I_{1},r_{1})})=\overline{G_{1}(D^{1}(I_{1},r_{1})) } .\]
 Since $\tilde{h}_{0}\in \overline{G_{1}(D^{1}(I_{1},r_{1})) }$, we have   
  \[ \tilde{h}_{1}=\tilde{G}_{1}^{-1}(\tilde{h}_{0})\in \overline{D^{1}(I_{1},r_{1})}\subseteq D^{0}(I_{1},r_{1}).\] 
  Inductively, we can prove  
  \[ \tilde{h}_{i}=\tilde{G}_{i}^{-1}\circ\cdots\circ\tilde{G}_{1}^{-1}(\tilde{h}_{0})\in D^{0}(I_{i},r_{i})\quad\text{for each }i\geq1.\]
  Take $\tilde{\textbf{\textit{h}}}=(\tilde{h}_{i})_{i\geq0}.$ It is clear that $\tilde{G}_{i+1}\tilde{h}_{i+1}=\tilde{h}_{i}$  for all $i\geq 0$.
  
The proof of the existence of $\hat{\textbf{\textit{h}}}$ is analogous and therefore we omit it.  
\end{proof}

Notice that $\tilde{h}_{0}$ is a limit of 
$C^{1}$-diffeomorphisms, which are $\varrho/4$-close to $I_{0}$ in the $C^{1}$-topology. Consequently, for each $x\in M_{0}$, \[ [\text{exp}_{x}^{-1}\circ\tilde{h}_{0}\circ \text{exp}_{x} - \text{exp}_{x}^{-1}\circ I_{0} \circ\text{exp}_{x}]|_{B(0_{x},\varrho)}\]
is $\varrho/4$-Lipschitz. 
Since  $\varrho<1$, we can prove that $\tilde{h}_{0}$ is injective. 
 Furthermore, for each $i\geq 0 $ and $x\in M_{i}$, we have \[\textbf{d}(\tilde{h}_{i} ^{-1}(x),x)=\textbf{d}(\tilde{h}_{i} ^{-1}(x),\tilde{h}_{i}   \tilde{h}_{i} ^{-1}(x))=\textbf{d}(y,\tilde{h}_{i}  (y)),\]
where $y={\tilde{h}_{i}} ^{-1}(x)$. Hence $d^{0}(\tilde{h}_{i},I_{i})=d^{0}(\tilde{h}_{i}^{-1},I_{i})$ for each $i\geq0$.

Analogously, we can prove that   $\hat{h}_{i}$ is invertible and $d^{0}(\hat{h}_{i},I_{i})=d^{0}(\hat{h}_{i}^{-1},I_{i})$ for each $i\leq 0$.

\begin{lem}\label{uniformlycontinuous} The families  $ (\tilde{h}_{i})_{i\geq0} $, $ (\tilde{h}_{i}^{-1})_{i\geq0} $,    $ (\hat{h}_{i})_{i\leq0}$ and $ (\hat{h}_{i}^{-1})_{i\leq0}$ are equicontinuous. 
\end{lem}
\begin{proof}Let $\varepsilon>0$. Since $\tilde{h}_{i}, \tilde{h}_{i}^{-1}\in D^{0}(I_{i},r_{i})$ and $r_{i}\rightarrow 0$ as $i\rightarrow+\infty,$ there exists $k>0$ such that, for each $i>k$, 
\[\max\{d^{0}(\tilde{h}_{i}, I_{i}), d^{0}(\tilde{h}_{i}^{-1}, I_{i})\}<\varepsilon/3 .\]
Hence, if $i<k$ and  $x,y\in M_{i}$ with $\textbf{d}(x,y)<\varepsilon/3$, then 
\[\textbf{d}(\tilde{h}_{i}(x),\tilde{h}_{i}(y))\leq \textbf{d}(\tilde{h}_{i}(x),I_{i}(x))+\textbf{d}(I_{i}(x),I_{i}(y))+\textbf{d}(I_{i}(y),\tilde{h}_{i}(y))<\varepsilon\]
and 
\[\textbf{d}(\tilde{h}_{i}^{-1}(x),\tilde{h}_{i}^{-1}(y))\leq \textbf{d}(\tilde{h}_{i}^{-1}(x),I_{i}(x))+\textbf{d}(I_{i}(x),I_{i}(y))+\textbf{d}(I_{i}(y),\tilde{h}_{i}^{-1}(y))<\varepsilon.\]

On the other hand, it is clear that there exists $\delta\in (0,\varepsilon/3)$ such that, if $0\leq i\leq k$, and $x,y\in M_{i} $ with $\textbf{d}(x,y)<\delta$, then \[\max\{\textbf{d}(\tilde{h}_{i}(x),\tilde{h}_{i}(y)),\textbf{d}(\tilde{h}_{i}(x)^{-1},\tilde{h}_{i}^{-1}(y))\}<\varepsilon. \]

The facts above prove that for each $i\geq 0$, if $x,y\in M_{i}$ and $\textbf{d}(x,y)<\delta$, then \[\max\{\textbf{d}(\tilde{h}_{i}(x),\tilde{h}_{i}(y)),\textbf{d}(\tilde{h}_{i}^{-1}(x),\tilde{h}_{i}^{-1}(y))\}<\varepsilon. \] 
Consequently,  $ (\tilde{h}_{i})_{i\geq0} $ and  $ (\tilde{h}_{i}^{-1})_{i\geq0} $  are equicontinuous. Analogously we can prove that     $ (\hat{h}_{i})_{i\leq0}$ and $ (\hat{h}_{i}^{-1})_{i\leq0}$ are equicontinuous.
\end{proof}

Finally, we have:
\begin{teo}\label{teoprinc} For all  $r\geq 1$, 
\[\textbf{H}:(\text{F}^{r}(\textbf{M}),\tau_{str})\rightarrow \mathbb{R}\cup \{+\infty\}  \quad\text{and}\quad\textbf{H}^{(-1)}:(\text{F}^{r}(\textbf{M}),\tau_{str})\rightarrow \mathbb{R}\cup \{+\infty\} \]
are  locally  constants. 
\end{teo}
\begin{proof} Let $\textbf{\textit{f}}\in \text{F}^{r}(\textbf{M})$. If follows from Lemmas \ref{existenciapuntofijo} and \ref{uniformlycontinuous} that there exists a strong basic neighborhood $B^{1}(\textbf{\textit{f}}, (r_{i})_{i\in\mathbb{Z}})$ such that  every  $\textbf{\textit{g}}\in B^{1}(\textbf{\textit{f}}, (r_{i})_{i\in\mathbb{Z}})$ is positively and negatively uniformly conjugate to $\textbf{\textit{f}}$. Thus, from  Theorem \ref{invarianteentropia} we have $\textbf{H}(\textbf{\textit{g}})=\textbf{H}(\textbf{\textit{f}})$ and $\textbf{H}^{(-1)}(\textbf{\textit{g}})=\textbf{H}^{(-1)}(\textbf{\textit{f}})$ for all $\textbf{\textit{g}}\in B^{1}(\textbf{\textit{f}}, (r_{i})_{i\in\mathbb{Z}})$, which proves the theorem. 
\end{proof}

 The author would like to thank the instituitions  Universidade de S\~ao Paulo (USP) and Instituto de Matem\'atica Pura e Aplicada (IMPA) and the agencies CAPES and CNPq for their   hospitality and support  during the course of the writing.

\end{document}